\documentclass[11pt]{amsart}
\usepackage{amsmath, amsthm, mathabx,amscd, amsfonts, amssymb, hyperref, mathrsfs, textcomp, epsfig, a4wide, graphicx, IEEEtrantools, tikz, verbatim, xypic}
\newcommand\numberthis{\addtocounter{equation}{1}\tag{\theequation}}
\usetikzlibrary{matrix, decorations.pathreplacing,angles,quotes}

\newtheorem{thm}{Theorem}

\newtheorem{prop}{Proposition}
\newtheorem{lemma}{Lemma}
\newtheorem{cor}{Corollary}
\theoremstyle{definition}

\theoremstyle{remark}
\newtheorem{remark}{Remark}
\newtheorem{example}{Example}

     \RequirePackage{rotating}                   
    \def\HSt{%
       \setbox0=\hbox{$\widehat{\mathit{HS}}$}
       \setbox1=\hbox{$\mathit{HS}$}
       \dimen0=1.1\ht0
       \advance\dimen0 by 1.17\ht1
       \smash{\mskip2mu\raise\dimen0\rlap{%
          \begin{turn}{180}
              {$\widehat{\phantom{\mathit{HS}}}$}
           \end{turn}} \mskip-2mu    
                \mathit{HS}
    }{\vphantom{\widehat{\mathit{HS}}}}{}}

     \RequirePackage{rotating}                   
    \def\HMt{%
       \setbox0=\hbox{$\widehat{\mathit{HM}}$}
       \setbox1=\hbox{$\mathit{HM}$}
       \dimen0=1.1\ht0
       \advance\dimen0 by 1.17\ht1
       \smash{\mskip2mu\raise\dimen0\rlap{%
          \begin{turn}{180}
              {$\widehat{\phantom{\mathit{HM}}}$}
           \end{turn}} \mskip-2mu    
                \mathit{HM}
    }{\vphantom{\widehat{\mathit{HM}}}}{}}

\newcommand{\HMf}{\widehat{\mathit{HM}}}

\newcommand{\spin}{\mathfrak{s}}

\newcommand{\Solv}{\mathsf{Solv}}
\newcommand{\X}{\mathcal{X}}
\newcommand{\Y}{\mathcal{Y}}
\newcommand{\Z}{\mathcal{Z}}

\begin{document}
\title{Monopole Floer homology and SOLV geometry}
\author{Francesco Lin}
\begin{abstract}We study the monopole Floer homology of a $\Solv$ rational homology sphere $Y$ from the point of view of spectral theory. Applying ideas of Fourier analysis on solvable groups, we show that for suitable $\Solv$ metrics on $Y$, small regular perturbations of the Seiberg-Witten equations do not admit irreducible solutions; in particular, this provides a geometric proof that $Y$ is an $L$-space. \end{abstract}
\maketitle

Among the three-dimensional model geometries, $\Solv$, i.e. $\mathbb{R}^3$ equipped with the metric $e^{2z}dx^2+e^{-2z}dy^2+dz^2$,
is the least symmetric one \cite{Sco}. This makes $\Solv$-manifolds (i.e. compact $3$-manifolds admitting a $\Solv$ metric) a very special class within the classification scheme of Thurston's geometrization theorem; if fact, they can be characterized as the geometric manifolds which are neither Seifert nor hyperbolic. From a historical perspective, their importance stems from the fact that many $\Solv$ manifolds arise as cusps of Hirzebruch modular surfaces \cite{Hir}; and the understanding of their signature defect was the main motivation behind the discovery of the Atiyah-Patodi-Singer index theorem for manifolds with boundary \cite{APS}, see \cite{ADS}. In a related fashion, three-dimensional $\Solv$ manifolds are also among the simplest examples where non-abelian Fourier analysis can be performed \cite{Bre}. More recently, the computation of their Heegaard Floer homology has provided evidence for the far-reaching $L$-space conjecture \cite{BCW}.
\par
In this paper we study the monopole Floer homology of a $\Solv$ rational homology sphere $Y$ from a geometric viewpoint. Monopole Floer homology is a package of invariants of three-manifolds introduced by Kronheimer and Mrowka in \cite{KM} obtained by studying the Seiberg-Witten equations (see also \cite{Lin3} for a friendly introduction). While monopole Floer homology is a topological invariant, and can be therefore computed in many cases using tools such as surgery exact triangles \cite{KMOS}, it is interesting to understand its relation with special geometric structures on the space, the case of Seifert fibered spaces \cite{MOY} being the prototypical example. In our case, a $\Solv$-rational homology sphere $Y$ has the structure of a torus semibundle, and admits several different $\Solv$-metrics obtained by rescaling the metrics along the fibers (see Section \ref{fourier} for a more detailed discussion of $\Solv$ geometry). Our main result is then the following.
\begin{thm}\label{main}
Let $Y$ be a $\Solv$-rational homology sphere, equipped with a $\Solv$ metric. If the fibers are small enough, then there are small regular perturbations for which the Seiberg-Witten equations on $Y$ do not admit irreducible solutions.
\end{thm}
The following is an immediate consequence of the theorem. Recall that a rational homology sphere $Y$ is an $L$-\textit{space} if $\HMf_*(Y,\spin)=\mathbb{Z}[U]$ as a $\mathbb{Z}[U]$-module for each spin$^c$ structure $\spin$.
\begin{cor}\label{corol}
Let $Y$ be a $\Solv$-rational homology sphere. Then $Y$ is an $L$-space.
\end{cor}
The analogous result in the setting of Heegaard Floer homology (which is known to yield isomorphic invariants, see \cite{HFHM1},\cite{CGH1} and subsequent papers) was proved by topological means in \cite{BCW} with $\mathbb{Z}/2\mathbb{Z}$-coefficients, and extended to $\mathbb{Z}$-coefficients in \cite{RR}. Let us also point out that compact $\Solv$ manifolds have either $b_1=0$ or $1$; in the latter case, they are Anosov torus bundles over the circle, and their Heegaard Floer homology (with $\mathbb{Z}$ coefficients) was computed in \cite{Bal}.
\\
\par
In our approach, we look at the monopole Floer homology of $\Solv$-manifolds from the point of view of spectral geometry. The main ingredient in the proof of Theorem \ref{main} is the following relation, for a rational homology sphere, between the existence of irreducible solutions to the Seiberg-Witten equations and the first eigenvalue $\lambda_1^*$ of the Hodge Laplacian on coexact $1$-forms (which improves on the main result of \cite{Lin4}).
\begin{thm}[Theorem $3$ of \cite{LL}]\label{spectral}
Let $Y$ be a rational homology sphere equipped with a metric $g$. Denote by $\tilde{s}(p)$ the sum of the two least eigenvalues of the Ricci curvature at the point $p$. If the inequality 
\begin{equation*}
\lambda_1^*\geq -\mathrm{inf}_{p\in Y}\tilde{s}(p)/2
\end{equation*}
holds, then the Seiberg-Witten equations do not admit irreducible solutions.
\end{thm}
In the case of a $\Solv$-metric, $\tilde{s}=-2$ at every point, so in order to prove Theorem \ref{main}, we need to show that for suitable $\Solv$-metrics on $Y$, $\lambda_1^*\geq 1$. Let us describe the strategy behind the proof of this by discussing the content of each section.
\par
In Section \ref{fourier}, we review some facts about the geometry and topology of $\Solv$-manifolds. As $\Solv$ is the left-invariant metric for a solvable Lie group structure on $\mathbb{R}^3$, one can study Fourier analysis on it, and we will introduce the basic ideas behind it. In Section \ref{spectralbound}, we use the aforementioned Fourier analysis to show that, for metrics with sufficiently small fibers, $\lambda_1^*= 1$, so that the Seiberg-Witten equations do not admit irreducible solutions by Theorem \ref{spectral}. As these metrics have $\lambda^*_1$ is exactly $1$, they lie in the borderline case of Theorem \ref{spectral}, and transversality is a quite subtle issue. We discuss it in Section \ref{trans}, where we will study explicit small perturbations of the equations and existence of harmonic spinors.

\vspace{0.3cm}
\textit{Acknowledgements. }The author would like to thank Liam Watson for some helpful comments. Thisf work was partially funded by NSF grant DMS-1807242.

\vspace{0.3cm}
\section{Compact Solvmanifolds and their Fourier analysis}\label{fourier}
We start by reviewing the basics of $\Solv$-geometry; most of the following discussion is taken from Section $12.7$ of \cite{Mar}. Recall that $\Solv$ is the Riemannian manifold $\mathbb{R}^3$ equipped with the metric
\begin{equation*}
e^{2z}dx^2+e^{-2z}dy^2+dz^2.
\end{equation*}
This is the left-invariant Riemannian metric on $\mathbb{R}^3$ when equipped with the solvable Lie group structure
\begin{equation*}
(x,y,z)\cdot(x',y',z')=(x+e^{-z}x', y+e^{z}y', z+z').
\end{equation*}
This can be though as the semidirect product corresponding to the splitting of
\begin{equation*}
0\rightarrow \mathbb{R}^2\rightarrow \Solv\stackrel{p}{\longrightarrow} \mathbb{R}\rightarrow 0,
\end{equation*}
where $p(x,y,z)=z$, given by
\begin{equation*}
z\mapsto \begin{bmatrix}e^z&0\\0 &e^{-z}\end{bmatrix}\in\mathrm{SL}(2,\mathbb{R}),
\end{equation*}
seen as linear automorphisms of $\mathbb{R}^2$. The Ricci tensor is given in this coordinates by
\begin{equation*}
\begin{bmatrix}0&0&0\\0 &0&0\\ 0&0&-2\end{bmatrix},
\end{equation*}
so that both $s$ and $\tilde{s}$ are $-2$ at each point. We can see that the foliation in $\mathbb{R}^2$ by the planes with $z$ constant descend to any compact $\Solv$-manifold; in fact, it descends to a foliation for which all the leaves are tori or Klein bottles.
\\
\par
Orientable compact solvmanifolds either have $b_1=0$ or $1$. The manifolds of the latter type, which will be denoted by $\tilde{Y}$, arise as quotients $\Gamma\setminus\Solv$ for lattices $\Gamma\subset \Solv$. Every such lattice is a split extension
\begin{equation*}
0\rightarrow \Lambda\rightarrow \Gamma\stackrel{p}{\longrightarrow} a\mathbb{Z}\rightarrow 0,
\end{equation*}
where $\Lambda\subset \mathbb{R}^2$ is a lattice invariant under the action of $\begin{bmatrix}e^a&0\\0 &e^{-a}\end{bmatrix}$. The underlying topological manifold is a torus bundle with monodromy $A\in \mathrm{SL}(2,\mathbb{Z})$; here $|\mathrm{tr}A|>2$ (i.e. $A$ is Anosov) and $e^a$ and $e^{-a}$ are its eigenvalues.
\begin{example}\label{figure8}
Consider $A=\begin{bmatrix}2&1\\1 &1\end{bmatrix}$. The mapping torus is well-known to be the zero surgery on the figure eight knot. Its eigenvalues are $\varphi^2$ and $\varphi^{-2}$ where $\varphi=\frac{1+\sqrt{5}}{2}$ is the golden ratio. Recall that it satisfies $\varphi^2=\varphi+1$. Consider the vectors
\begin{equation*}
v=(\varphi,1-\varphi)\quad w=(1,1).
\end{equation*}
If $S$ is the matrix with colums $v$ and $w$, we have $A=S^{-1}\begin{bmatrix}\varphi^2&0\\0 &\varphi^{-2}\end{bmatrix}S$; setting $\Lambda$ to be the lattice generated by $v$ and $w$, and $a=\mathrm{log}(\varphi^2)$, we obtain the lattice $\Gamma$ equipping the mapping torus of $A$ with a $\Solv$ metric.
\end{example}
\begin{remark}\label{NT}
We can also think about this example from a more number theoretic viewpoint, which makes the connection with \cite{Hir} and \cite{ADS} clearer. Consider the field $k=\mathbb{Q}(\sqrt{5})$. It is totally real, and it comes with two natural embeddings $\phi_+,\phi_-$ into $\mathbb{R}$ sending $\sqrt{5}$ to $\pm\sqrt{5}$. The ring of integers $\mathcal{O}_k$ is the lattice $\Lambda=\mathbb{Z}[\varphi]$ which has basis $\varphi$ and $1$. The group of totally positive units is generated by $\varphi^2$; and it is easy to see that its multiplication action is given in our chosen basis by $A$. Finally, we can embed the lattice $\Lambda$ in $\mathbb{R}^2$ using $(\phi_+,\phi_-)$; our basis elements are mapped to the vectors $v$ and $w$.
\end{remark}
A $\Solv$-manifold with $b_1=0$, denoted by $Y$, is a torus semibundle; therefore it admits a double cover $\tilde{Y}$ which is a $\Solv$ torus bundle $\Gamma\setminus\Solv$. Then $Y$ can be described in the following way. For a choice of basis $v,w$ of the lattice $\Lambda=\Gamma\cap \mathbb{R}^2$, with corresponding left-invariant extension $\mathcal{V},\mathcal{W}$, we can consider the additional orientation-preserving isometry of $\tilde{Y}$ sending
\begin{equation*}
(a\mathcal{V}+b\mathcal{W},z)\mapsto ((a+\frac{1}{2})\mathcal{V}-b\mathcal{W},-z).
\end{equation*}
In particular, the action on the fiber $z=0$ (which is preserved) is obtained by $(av+bw)\mapsto (a+\frac{1}{2})v-bw$; and the quotient of the fiber is a Klein bottle. This is an order $2$ isometry $\tilde{Y}$, and the quotient is $Y$.
\\
\par
From this description, we see that on any $\Solv$-manifold $Y$we obtain a one parameter family of metrics obtained by rescaling the lattice $\Lambda$; this can be seen concretely in Example \ref{figure8}.
\\
\par
Let us now introduce the basics of Fourier analysis on a compact Solvmanifold with $b_1(Y)=1$. We follow the first chapter of \cite{Bre}, to which we refer for a pleasant, more thorough, discussion.
\par
Consider a smooth function $f:\Gamma\setminus\Solv\rightarrow\mathbb{R}$. This can be thought (with a little abuse of notation) as a function $f:\Solv\rightarrow \mathbb{R}$ which is left invariant under $\Gamma$. In particular, it is invariant under the action of $\Lambda\subset\Gamma$, i.e.
\begin{equation*}
f(\underline{x}+m,z)=f(\underline{x},z)\text{ for all }m\in \Lambda.
\end{equation*}
We can therefore expand $f$ in Fourier series in the $\mathbb{R}^2\times\{0\}\subset\Solv$ directions
\begin{equation*}
f(\underline{x},z)=\sum_{\mu\in \Lambda'} a_{\underline{\mu}}(z)e^{ i \underline{\mu}\cdot \underline{x}}.
\end{equation*}
for some smooth functions $a_{\underline{\mu}}(z)$. Here $\Lambda'$ is the dual lattice of $\Lambda$, where we use the convention
\begin{equation*}
\Lambda'=\{\underline{\mu}\in \mathbb{R}^2\lvert \underline{\mu}\cdot m\in 2\pi\mathbb{Z}\text{ for all }m\in\Lambda\}.
\end{equation*}
We now use the fact that $f$ is invariant by the action of $(\underline{0},a)$. Letting $A=\begin{bmatrix}e^a&0\\0 &e^{-a}\end{bmatrix}$, we see that
\begin{equation*}
f(\underline{x},z)=f((\underline{0},a)\cdot(\underline{x},z))=f(A\underline{x}, z+a),
\end{equation*}
hence, after reindexing,
\begin{equation*}
\sum_{\underline{\mu}\in \Lambda'} a_{\underline{\mu}}(z)e^{ i \underline{\mu}\cdot \underline{x}}=\sum_{\underline{\mu}\in \Lambda'} a_{\underline{\mu}}(z+a)e^{ i \underline{\mu}\cdot A\underline{x}}=\sum_{\underline{\mu}\in \Lambda'} a_{\underline{\mu}\cdot A}(z+a)e^{ i \underline{\mu}\cdot\underline{x}}.
\end{equation*}
This implies that
\begin{equation*}
a_{\underline{\mu}}(z)=a_{\underline{\mu}\cdot A}(z+a),
\end{equation*}
so $a_{\underline{\mu}}$ determines via translation $a_{\underline{\mu}\cdot A^n}$. In particular, the Fourier series is determined by the collection of functions for $a_{\underline{\mu}}(z)$ for $\underline{\mu}\in \Lambda'/V$, $V$ being the group of automorphisms of the dual lattice $\Lambda'$ generated by $A$. While $a_0$ is a periodic function with period $a$, it can be shown that the functions $a_{\underline{\mu}}(z)$ for $\underline{\mu}\neq 0$ are in the Schwartz-type space
\begin{equation}\label{schwartz}
\mathcal{S}=\{f\lvert e^{nz}f^{(m)}(z) \text{ is bounded for all }n\in\mathbb{Z},m\geq0\},
\end{equation}
where $f^{(m)}$ denotes the $m$th derivative of $f$.
\\
\par
With this in mind, let us study as a warm-up example the Laplacian on functions on $\Gamma\setminus\Solv$, which can be written as
\begin{equation*}
\Delta f=-(e^{-2z}f_{xx}+e^{+2z}f_{yy}+f_{zz}).
\end{equation*}
Let us use the decomposition in Fourier modes discussed above. We then have a $L^2$-unitary decomposition
\begin{equation*}
\Delta=\bigoplus_{\underline{\mu}\in M'/V} \Delta_{\underline{\mu}},
\end{equation*}
where $\Delta_{0}$ acts on $L^2(\mathbb{R}/a\mathbb{Z})$ and $\Delta_{\underline{\mu}}$ is a diagonalizable operator on  $L^2(\mathbb{R})$. In particular, if we have $\underline{\mu}=(\mu,\mu')$, the corresponding operator is given by substituting
\begin{equation*}
\frac{d}{dx}\mapsto  i\mu,\quad\frac{d}{dy}\mapsto i\mu'
\end{equation*}
so that
\begin{equation*}
\Delta_\mu f=-f_{zz}+(\mu^2 e^{-2z}+(\mu')^2e^{2z})f.
\end{equation*}
Therefore $\lambda$ is an eigenvalue of $\Delta_{\underline{\mu}}$ if and only if
\begin{equation*}
f_{zz}=(\mu^2 e^{-2z}+(\mu')^2e^{2z}-\lambda)f.
\end{equation*}
While this equation is not solvable in terms of elementary functions, we can still understand the basic properties of its spectrum. Let us first recall the following well-known elementary lemma.
\begin{lemma}\label{keyODE}
Suppose $f:\mathbb{R}\rightarrow\mathbb{R}$ solves the second order linear ODE
\begin{equation*}
f_{zz}=\Phi(z)\cdot f
\end{equation*}
where $\Phi$ is smooth and $\Phi(z)>0$ everywhere. Then $f$ cannot be in $L^2(\mathbb{R})$.  
\end{lemma}

\begin{proof}
Possibly after replacing $f(z)$ by $-f(z)$ or $f(-z)$, we can assume that at $x_0$ both $f(x_0)=c>0$ and $f'(x_0)\geq0$. Suppose there is $t_0>x_0$ with $0<f(t_0)<f(x_0)$. We can also assume $f>0$ on $[x_0,t_0]$. Then there is $x_0<t<t_0$ with $f'(t)<0$. Applying again the mean value theorem, there is $x_0<t'<t$ with $f''(t')<0$, which is contradiction as $f''(t')=\Phi(z)\cdot f>0$. So $f(x)\geq c$ for $x\geq x_0$, and the result follows.
\end{proof}

We then have the following.

\begin{lemma}\label{eigenvalue}
For $\underline{\mu}\neq0$ the first eigenvalue of $\Delta_{\underline{\mu}}$ is at least $2|\mu\mu'|\neq0$.
\end{lemma}
\begin{proof}
By \textsc{AM-GM}, the inequality
\begin{equation*}
\mu^2 e^{-2z}+(\mu')^2e^{2z}\geq 2|\mu\mu'|, 
\end{equation*}
holds, and the result follows from the previous lemma. 
\end{proof}
In terms of the number theoretic description in Remark \ref{NT}, the quantity $\mu\mu'$ is the \textit{norm} $N(\underline{\mu})$; the only basic property we will need is that there is $c>0$ such that $|\mu\mu'|\geq c$ for all $\underline{\mu}\in \Lambda'\setminus \{0\}$.
\\
\par
For completeness, let us conclude this section by discussing the zero mode $\underline{\mu}=0$. In this case, we study the ODE
\begin{equation*}
f_{zz}=-\lambda f
\end{equation*}
with $f$ periodic with period $a$. It has eigenvalues $\lambda=\frac{4\pi^2}{a^2} n^2$ for $n\in\mathbb{Z}$.
\vspace{0.3cm}
\section{The spectrum on coexact $1$-forms}\label{spectralbound}
In this section we will perform the key computation behind our main result. Recall from the previous section that on a $\Solv$-manifold there is a non-trivial family of metrics obtain by rescaling the lattice $\Lambda\subset \mathbb{R}^2$. With is in mind, we have the following.
\begin{prop}\label{maineigen}
Let $Y$ be a rational homology sphere equipped with a $\Solv$ metric such that the fibers are small enough. Then the first eigenvalue on coexact $1$-forms satisfies $\lambda_1^*=1$. Furthermore, the $1$-eigenspace is one dimensional.
\end{prop}
In fact, our proof will provide an explicit smallness condition for the fibers.
\\
\par
Let us start by considering the the case of a $\Solv$-manifold $\tilde{Y}=\Gamma\setminus \Solv$ with $b_1=1$. The $1$-forms
\begin{equation*}
\X=e^zdx,\quad \Y=e^{-z}dy\quad \Z=dz
\end{equation*}
descend to a left-invariant dual orthonormal frame on $\tilde{Y}$. We can then write any $1$-form $\xi$ as
\begin{equation*}
\xi=f\X+g\Y+h\Z,
\end{equation*}
where $f,g,h$ are functions on $\Gamma\setminus\Solv$, or equivalently left-invariant functions on $\Solv$. We are interested in understanding for which $\lambda$ the equation
\begin{equation*}
\ast d\xi=\lambda\xi
\end{equation*}
admits non-trivial solutions. Notice that, provided $\lambda\neq0$, such a form necessarily satisfies $d\ast\xi=0$, i.e. it is coclosed. We have\begin{align*}
d\xi&=(e^{-z}g_x-e^zf_y)\X\wedge \Y+\\
&+(-g_z+g+e^zh_y)\Y\wedge \Z+\\
&+(f_z+f-e^{-z}h_x)\Z\wedge \X
\end{align*}
so that our equation is equivalent to the system
\begin{align*}
-g_z+g+e^zh_y&=\lambda f\\
f_z+f-e^{-z}h_x&=\lambda g\\
e^{-z}g_x-e^zf_y&=\lambda h,\numberthis \label{eqn}
\end{align*}
while coclosedness is equivalent to
\begin{equation*}
e^{-z}f_x+e^zg_y+h_z=0.
\end{equation*}
Differentiating  we get
\begin{align*}
-e^{-2z}h_{xx}&=-e^{-z}f_{xz}-e^{-z}f_x+\lambda e^{-z}g_x\\
-e^{2z}h_{yy}&=-e^{z}g_{yz}+e^z g_y-\lambda e^{z}f_y\\
-h_{zz}&=e^{-z}f_{xz}-e^{-z}f_x+e^zg_{yz}+e^zg_y,
\end{align*}
therefore summing we obtain
\begin{equation*}
\Delta h=\lambda^2 h-2e^{-z}f_x+2e^zg_y,
\end{equation*}
where $\Delta$ denotes the Laplacian on functions on $\tilde{Y}$.
Similarly for $g$ we obtain
\begin{align*}
-e^{-2z}g_{xx}&=-\lambda e^{-z} h_x-f_{xy}\\
-e^{2z}g_{yy}&=f_{xy}+e^zh_{yz}\\
-g_{zz}&=\lambda f_z-g_z-e^zh_y-e^zh_{yz},
\end{align*}
hence summing
\begin{align*}
\Delta g&=\lambda(f_z-e^{-z}h_x)-g_z-e^zh_y=\lambda^2g-\lambda f-g_z-e^zh_y=\\
&=(\lambda^2-1)g-2e^z h_y.
\end{align*}
Finally, as
\begin{align*}
-e^{-2z}f_{xx}&=e^{-z}h_{xz}+g_{xy}\\
-e^{2z}f_{yy}&=\lambda e^z h_y-g_{xy}\\
-f_{zz}&=f_z-e^{-z}h_{xz}+e^{-z}h_x-\lambda g_z,
\end{align*}
we have
\begin{align*}
\Delta f&=\lambda(-g_z+e^zh_y)+f_z+e^{-z}h_x=\lambda^2 f-\lambda g+f_z+e^{-z}h_x=\\
&=(\lambda^2-1)f+2e^{-z}h_x.
\end{align*}
Notice that $\Z$ is a harmonic $1$-form; as $b_1=1$, all harmonic forms are multiples of it.
\begin{lemma}\label{cover}
Le $\tilde{Y}$ be a $\Solv$ manifold with $b_1=1$ equipped with a metric for which the fibers are small enough. Then $\lambda_1^*=1$, and the $1$-eigenspace is spanned by $X$ and $Y$.
\end{lemma}
\begin{proof}
We can expand $f,g$ and $h$ in Fourier series; the operator $\ast d$ decomposes accordingly in the sum of $\ast d_{\underline{\mu}}$, and in the $\underline{\mu}$ component our equations look like
\begin{align*}
\Delta_{\mu}h&=\lambda^2 h-2i\mu e^{-z} f+2i\mu'e^z g\\
\Delta_{\mu}g&=(\lambda^2-1)g-2i\mu'e^zh\\
\Delta_{\mu}f&=(\lambda^2-1)f+2i\mu e^{-z}h
\end{align*}
with $f,g$ and $h$ are \textit{complex} valued functions in the space $\mathcal{S}$.
\par
Let us discuss first the modes $\underline{\mu}\neq0$. By Lemma \ref{eigenvalue}, the bottom of the spectrum of $\Delta_{\underline{\mu}}$ is bounded below by $2|\mu\mu'|$; and furthermore, by suitably rescaling the metric, we can arrange that this quantity is $>16$ for all $\underline{\mu}\neq0$. Multiplying each equation by $\bar{h},\bar{g}$ and $\bar{f}$ respectively, and adding them together, we obtain the pointwise identity
\begin{equation*}
\bar{h}\Delta_{\mu}h+\bar{g}\Delta_{\mu}g+\bar{f}\Delta_{\mu}f= \lambda^2|h|^2+(\lambda^2-1)|g|^2+(\lambda^2-1)|f|^2+4\mathrm{Re}(i\mu'e^zg\bar{h})-4\mathrm{Re}(2i\mu e^{-z}f\bar{h}).
\end{equation*}
In particular, this implies that the left-hand side is real. By the Peter-Paul inequality, we have the pointwise inequalities
\begin{align*}
|4\mathrm{Re}(i\mu'e^zg\bar{h})|&\leq4|\mu'e^z\bar{h}||g|\leq \frac{(\mu')^2e^{2z}}{2}|h|^2+8|g|^2\\
|4\mathrm{Re}(i\mu e^{-z}f\bar{h})|&\leq4|\mu e^{-z}\bar{h}||f|\leq\frac{\mu^2e^{-2z}}{2}|h|^2+8|f|^2
\end{align*}
so that
\begin{equation}\label{ineq}
\bar{h}\tilde{\Delta}_{\mu}h+\bar{g}\Delta_{\mu}g+\bar{f}\Delta_{\mu}f\leq \lambda^2|h|^2+(\lambda^2+7)|g|^2+(\lambda^2+7)|f|^2
\end{equation}
where
\begin{equation*}
\tilde{\Delta}_{\mu}h=-h_{zz}+\frac{1}{2}(\mu^2 e^{-2z}+(\mu')^2e^{2z})h
\end{equation*}
is still a diagonalizable operator over $L^2(\mathbb{R})$. The same argument as Lemma \ref{eigenvalue} implies that the first eigenvalue of $\tilde{\Delta}_{\mu}$ is at least $|\mu\mu'|$. Therefore, by integrating the inequality (\ref{ineq}) we have
\begin{equation*}
|\mu\mu'|(\|h\|^2+\|g\|^2+\|f\|^2)\leq (\lambda^2+7)(\|h\|^2+\|g\|^2+\|f\|^2).
\end{equation*}
As by assumption $|\mu\mu'|>8$, $\lambda^2>1$.
\par
Finally, we deal with the zero mode. Suppose $0<\lambda^2< 1$. Then $\lambda^2-1<0$, hence
\begin{equation*}
-g_{zz}=(\lambda^2-1)g,\quad -f_{zz}=(\lambda^2-1)f
\end{equation*}
have no periodic solution. It follows from Equation (\ref{eqn}) that $h$ is constant, so we have a multiple of the harmonic form $\Z$. Finally, the case $\lambda^2=1$ corresponds to the span of $\X$ and $\Y$.
\end{proof}

Finally, we are ready to prove Proposition \ref{maineigen}.
\begin{proof}[Proof of Proposition \ref{maineigen}]
Suppose $Y$ is a $\Solv$-rational homology sphere. Consider its double cover $\pi:\tilde{Y}\rightarrow Y$ where $\tilde{Y}$ has $b_1(\tilde{Y})=1$. If $\xi$ is a $\lambda$-eigenform on $Y$, the $\pi^*\xi$ is a $\lambda$-eigenform on $\tilde{Y}$. Choose a $\Solv$-metric with fibers small enough, so that Lemma \ref{cover} applies. This implies that on $Y$ we have $\lambda_1^*\geq1$, and furthermore that if $\xi$ is a $1$-eigenform on $Y$, then $\pi^*\xi$ is a linear combination of $\X$ and $\Y$. Finally, in the notation of Section \ref{fourier}, if $v,w$ is the basis of $\Lambda$, then exactly the linear combinations of $\X$ and $\Y$ that vanish on $w$ at $z=0$ descend to $Y$.
\end{proof}

We will denote by $\eta$ the unique the unit length $1$-eigenforms such that $\eta(v)>0$ and $\eta$ descends to $Y$. Recall (Chapter $28$ of \cite{KM}) that there is a natural one-to-one correspondence between spin$^c$ structures unit length $1$-forms up to homotopy outside balls. With this in mind, we have the following.
\begin{lemma}\label{spinstr}
The unit length $1$-form $\eta$ determines a spin structure $\spin_0$ on $Y$.
\end{lemma}
\begin{proof}
Denote by $\zeta$ the unit length $1$-form obtained from $\eta$ by a couterclockwise rotation of $\pi/2$ within the fibers of $\tilde{Y}$. Then $\eta,\zeta$ and $\Z$ form a dual orthonormal frame of $\tilde{Y}$. By twisting $\zeta$ and $\Z$ around $\eta$ on $\mathbb{R}^2\times\{0\}$, so that under translation by $v/2$ they go into their opposite, and extending in a left-invariant fashion, we obtain a new frame $\eta,\zeta'$ and $\Z'$ of $\Solv$ that descends to $Y$. 
The spin$^c$ structure conjugate to $\spin_0$ corresponds to the $1$-form $-\eta$; but $\eta$ and $-\eta$ are homotopic on $Y$ through $\cos(t)\eta+\sin(t)\zeta'$ for $t\in[0,2\pi]$.
\end{proof}

\vspace{0.3cm}
\section{Transversality}\label{trans}
In the previous section, we have exhibited a metric for which $\lambda_1^*=-\mathrm{inf} (\tilde{s}/2)$. As this is the borderline case of Theorem \ref{spectral}, transversality is a quite delicate issue as small perturbation might introduce irreducible solutions. This should be compared with the discussion of flat manifolds in Chapter $37$ of \cite{KM}. As in their setting, we will show that we can achieve transversality, while still not having irreducible solutions, by considering the perturbed functional
\begin{equation*}
\mathcal{L}(B,\Psi)-\frac{\delta}{2}\|\Psi\|^2
\end{equation*}
for $\delta$ sufficiently small. The corresponding equations for the critical points are
\begin{align*}
D_B\Psi&=\delta\Psi\\
\frac{1}{2}\rho(F_{B^t})&=(\Psi\Psi^*)_0.
\end{align*}
We have the following.
\begin{lemma}
Consider a spin$^c$ structure $\spin\neq \spin_0$. Then, for $\delta$ small enough, the perturbed Seiberg-Witten equations do not admit irreducible solutions.
\end{lemma}
\begin{proof}Suppose we have a sequence $\delta_i\rightarrow 0$ with corresponding irreducible solutions $(B_i,\Psi_i)$; consider the corresponding configurations in the blow-up $(B_i,s_i,\psi_i)$, where $\|\psi_i\|_{L^2}=1$. These admit (up to gauge transformations, and up to passing to a subsequence) a limit $(B,s,\psi)$ which solves the blown-up equations with $\delta=0$; in particular, as the unperturbed equations do not admit irreducible solutions by Theorem \ref{spectral}, $s=0$, $B$ is the flat connection, and $D_B\psi=0$. Recall that, setting $\xi=\rho^{-1}(\Psi\Psi^*)_0$, it is shown in \cite{LL} that for solutions $(B,\Psi)$ of the \textit{unperturbed} Seiberg-Witten equations the pointwise identity
\begin{equation*}
|\nabla\xi|^2+|d\xi|^2=|\Psi|^2|\nabla_B\Psi|^2
\end{equation*}
holds. This holds for the perturbed equations up to an error going to zero for $\delta_i\rightarrow 0$; hence it will apply to the limit form $\alpha=\rho^{-1}(\psi\psi^*)_0$. Furthermore, as it is the limit of the sequence of coexact forms $\frac{1}{s_i^2}\frac{1}{2}\rho(F_{B^t})$, $\alpha$ is a coexact $1$-form.
\par
Let us study the geometry of $\alpha$. As $\psi$ is a harmonic spinor, and $B$ is flat, the Weitzenb\"ock formula on $Y$ implies
\begin{equation*}
\nabla_B^*\nabla_B \psi=\frac{1}{2}\psi,
\end{equation*}
hence the pointwise identity
\begin{equation*}
\Delta|\psi|^2=2\langle \psi, \nabla_B^*\nabla_B \psi\rangle-2|\nabla_B\psi|^2=| \psi|^2-2|\nabla_B\psi|^2
\end{equation*}
holds. Multiplying by $|\psi|^2$ and integrating, we obtain
\begin{equation*}
\int | \psi|^4-\int  2| \psi|^2|\nabla_B\psi|^2=\int  | \psi|^2\Delta|\psi|^2\geq0.
\end{equation*}
Recalling now that $|\alpha|^2=\frac{1}{4}|\psi|^4$, we obtain, by using the Bochner formula and $\lambda_1^*=1$, the chain of inequalities
\begin{align*}
2\|\alpha\|^2_{L^2}=\int\frac{1}{2}|\psi|^4\geq \int  | \psi|^2|\nabla_B\psi|^2&=\|\nabla\alpha\|^2_{L^2}+\|d\alpha\|^2_{L^2}\\
 &=2\|d\alpha\|^2_{L^2}-\mathrm{Ric}(\alpha,\alpha)\geq2\|d\alpha\|^2_{L^2}\geq 2\|\alpha\|^2_{L^2}.
\end{align*}
This implies that all inequalities are equalities, so that in particular $\alpha$ is a $1$-eigenform, i.e. a multiple of $\eta$. Finally, by Lemma \ref{spinstr} this can happen if and only if the underlying spin$^c$ structure is the spin structure $\spin_0$.
\end{proof}

We need to understand more in detail the spin structure $\spin_0$ on $Y$; before doing this, let us study the spin geometry of the double cover $\tilde{Y}$. The manifold $\tilde{Y}=\Gamma\setminus\mathrm{Solv}$ comes with a natural spin structure $\spin_*$ coming from the left invariant orthonormal framing dual to $\Z,\X,\Y$, i.e.
\begin{equation*}
e_1=\frac{d}{dz},\quad e_2=e^{-z}\frac{d}{dx},\quad e_3=e^z\frac{d}{dy}.
\end{equation*}
This defines a spin structure $\spin_*$ by taking the trivial bundle $S=Y\times \mathbb{C}^2$ and letting these vector fields act via the Pauli matrices
\begin{equation*}
\begin{bmatrix}
i&0\\
0&-i
\end{bmatrix}
\quad
\begin{bmatrix}
0&-1\\
1&0
\end{bmatrix}
\quad
\begin{bmatrix}
0&i\\
i&0
\end{bmatrix}.
\end{equation*}
Let $B_*$ the spin connection on $Y$ induced by the Levi-Civita connection.

\begin{lemma}\label{harmonic}
The kernel of the Dirac operator $D_{B_*}$ consists of the constant spinors. 
\end{lemma}
\begin{proof}
Let us write explicitly the Dirac operator. Our orthonormal frame satisfies the commutation relations
\begin{align*}
[e_1,e_2]&=-e_2\\
[e_1,e_3]&=e_3\\
[e_2,e_3]&=0.
\end{align*}
Setting $[e_i,e_j]=\sum_k C_{ijk} e_k$, we have that the Christoffel symbols are
\begin{equation*}
\Gamma_{ijk}=\frac{1}{2}(C_{ijk}-C_{ikj}-C_{jki}),
\end{equation*}
hence in our case the non-zero ones are
\begin{equation*}
\Gamma_{212}=-\Gamma_{221}=1,\quad \Gamma_{313}=-\Gamma_{331}=-1.
\end{equation*}
The spin connection on the spinor bundle is given by
\begin{equation*}
\nabla_{e_i}\Psi=e_i(\Psi)+\frac{1}{4}\sum_{j<k}\Gamma_{ijk}[\sigma_j,\sigma_k]\cdot \Psi,
\end{equation*}
see Section $3.3$ of \cite{BGV}. Therefore, as $[\sigma_1,\sigma_2]=-2\sigma_3$ and $[\sigma_1,\sigma_3]=2\sigma_2$, we have
\begin{align*}
\nabla_{e_1}\Psi&=e_1(\Psi)\\
\nabla_{e_2}\Psi&=e_2(\Psi)-\frac{1}{2}\sigma_3\cdot\Psi\\
\nabla_{e_3}\Psi&=e_3(\Psi)-\frac{1}{2}\sigma_2\cdot\Psi
\end{align*}
As $\sigma_2$ and $\sigma_3$ anticommute, we have
\begin{equation*}
D_{B_*}\Psi=\sum_i\rho(e_i)\cdot\nabla_{e_i}\Psi=\sum_i\rho(e_i)\cdot{e_i}(\Psi).
\end{equation*}
Hence, writing $\Psi=(f, g)$, we have
\begin{equation*}
D_{B_*}\begin{bmatrix}f\\g\end{bmatrix}=
\begin{bmatrix}
if_z-e^{-z}g_x+ie^zg_y\\
-ig_z+e^{-z}f_x+ie^zf_y
\end{bmatrix},
\end{equation*}
and the equations for a harmonic spinor are
\begin{align*}
f_z+ie^{-z}g_x+e^zg_y&=0\\
g_z+ie^{-z}f_x-e^zf_y&=0.
\end{align*}
Let us now decompose the equations according to the eigenmodes $\underline{\mu}\in\Lambda'$. We obtain
\begin{align*}
f_z-\mu e^{-z}g+i\bar{\mu}e^zg&=0\\
g_z-\mu e^{-z}f-i\bar{\mu}e^zf&=0.
\end{align*}
Of course for the zero mode the kernel consists of constant solutions. Let us show now that the eigenmodes with $\underline{\mu}\neq 0$ do not admit non-zero harmonic spinors. We have
\begin{align*}
\frac{d}{dz}|f|^2&=2\mathrm{Re}(f_z\bar{f})=2\mathrm{Re}((\mu e^{-z}g-i\bar{\mu}e^zg)\bar{f})\\
\frac{d}{dz}|g|^2&=2\mathrm{Re}(\bar{g}_z{g})=2\mathrm{Re}((\mu e^{-z}\bar{f}-i\bar{\mu}e^z\bar{f})g)
\end{align*}
hence
\begin{equation*}
\frac{d}{dz}(|f|^2-|g|^2)=0.
\end{equation*}
As $|f|^2-|g|^2$ is in the class of function $\mathcal{S}$ from equation (\ref{schwartz}), we have $|f|^2=|g|^2$ everywhere. This, together with our ODE, shows that the functions $f$ and $g$ are never zero. We then have
\begin{align*}
f_z\bar{g}&=\mu e^{-z}|g|^2-i\bar{\mu}e^{z}|g|^2\\
f\bar{g}_z&=\mu e^{-z}|f|^2-i\bar{\mu}e^{z}|f|^2
\end{align*}
hence
\begin{equation*}
\frac{d}{dz}(\frac{f}{\bar{g}})=\frac{f_z\bar{g}-f\bar{g}_z}{\bar{g}^2}=0.
\end{equation*}
Therefore, up to multiplying $f$ and $g$ by the same complex constant, we have 
\begin{equation*}
g=\bar{f}
\end{equation*}
and both are equations are equivalent to
\begin{equation*}
f_z=\mu e^{-z}\bar{f}+i\bar{\mu}e^{z}\bar{f}.
\end{equation*}
Writing $f=a+ib$ for real functions $a,b$, this can be written as the system
\begin{align*}
a_z&=\mu e^{-z}a+\bar{\mu}e^zb\\
b_z&=\bar{\mu}e^za-\mu e^{-z}b
\end{align*}
Differentiating the first equation, and making some simple substitutions, we obtain the equation
\begin{equation*}
a_{zz}=a_z+(\mu^2e^{-2z}+\bar{\mu}^2e^{2z}-2\mu e^{-z})a.
\end{equation*}
Then, $A=e^{-z/2}a$ (which still lies in $\mathcal{S}$) satisfies an equation of the form $A_{zz}=\Psi\cdot A$ where, for our choice of $\Solv$ metric, $\Psi>0$ everywhere. Again by Lemma \ref{keyODE}, $A$ is zero, and so are $a$ and $b$.
\end{proof}

With this computation in mind, we will show that the Dirac operator on our rational homology sphere $Y$ equipped with the spin structure $\spin_0$ has no kernel by suitably pulling back the spin structure along finite covers, and applying Lemma \ref{harmonic}.
\par
First of all, we pull it back to $\tilde{Y}$; suppose that this is the mapping torus of $A\in\mathrm{SL}(2;\mathbb{Z})$. Every element in $A\in\mathrm{SL}(2;\mathbb{Z}/2\mathbb{Z})$ has order $6$ so that $A^6=\mathrm{Id}$ modulo $2$; the mapping torus of $A^6$, call it $\overline{Y}$, admits a degree $6$ covering map $p:\overline{Y}\rightarrow\tilde{Y}$. The Mayer-Vietors sequence for the mapping torus of any map $f$ implies the exact sequence
\begin{equation*}
H_1(T^2;\mathbb{Z}/2\mathbb{Z})\stackrel{1-f_*}{\rightarrow} H_1(T^2;\mathbb{Z}/2\mathbb{Z})\rightarrow H_1(M_f;\mathbb{Z}/2\mathbb{Z})\rightarrow \mathbb{Z}/2\mathbb{Z}\rightarrow 0.
\end{equation*}
In our case, this implies that
\begin{equation*}
H^1(\overline{Y};\mathbb{Z}/2\mathbb{Z})\cong \mathbb{Z}/2\mathbb{Z}\oplus H^1(T^2;\mathbb{Z}/2\mathbb{Z})\cong (\mathbb{Z}/2\mathbb{Z})^3,
\end{equation*}
so that, from the point of view of spin topology, $\overline{Y}$ looks like the more familiar three-torus. From the description in Lemma \ref{spinstr}, it readily follows that the pullback of $\spin_0$ to $\overline{Y}$, call it $\overline{\spin}$, is the spin structure obtain from the standard one $s_*$ by twisting by $2\pi$ around the class dual to $v$ in $H^1(T^2;\mathbb{Z}/2\mathbb{Z})$ (which is a non-trivial operation). The sublattice of $\Lambda$ spanned by $2v$ and $w$ is preserved by $A^6$; the corresponding mapping torus $\overline{\overline{Y}}$ is a double cover of $\overline{Y}$; and the pullback of $\overline{s}$ is the standard spin structure $\spin_*$ on $\overline{\overline{Y}}$. One can then identify the harmonic spinors on $(\overline{Y},\overline{\spin})$ as the harmonic spinors on $(\overline{\overline{Y}},\spin_*)$ which change sign under translation by $v$; by Lemma \ref{harmonic}, there are no such spinors. Hence, there are no harmonic spinors on the base space $(Y,\spin_0)$.
\\
\par
Putting pieces together, we finally conclude.
\begin{proof}[Proof of Theorem \ref{main}]By the discussion above, we have found small perturbations for which there are no irreducible solutions and the (perturbed) Dirac operator of the reducible solution has no kernel; we can then add a further small perturbation to make all of its eigenvalues simple (while preserving these properties) as in Chapter $12$ of \cite{KM}; the proof of Theorem \ref{main} is then completed.
\end{proof}
\vspace{0.5cm}

\bibliographystyle{alpha}
\bibliography{biblio}

\end{document}